\documentclass[]{amsart}
\usepackage[utf8]{inputenc}
\usepackage{hyperref}
\usepackage{amsmath}
\usepackage{amsfonts}
\usepackage{amsthm}
\usepackage{amssymb}
\usepackage{color}
\usepackage{todonotes}
\usepackage{comment}
\usepackage{tikz}
\usepackage{hvfloat}
\usepackage{enumerate}
\usepackage{todonotes}
\usetikzlibrary{decorations.markings}
\usetikzlibrary{arrows}
\usepackage{graphicx}
\usepackage{caption}
\usepackage{subcaption}
\usepackage{mathtools}
\usepackage[all]{xy}

\usepackage{xcolor}

\makeatletter
\renewcommand{\boxed}[1]{\text{\fboxsep=.2em\fbox{\m@th$\displaystyle#1$}}}
\makeatother

\newcommand{\Z}{\mathbb{Z}}

\newcommand{\N}{\mathbb{N}}

\renewcommand{\le}{\leqslant}
\renewcommand{\ge}{\geqslant}

\theoremstyle{plain}% default
\newtheorem{thm}{Theorem}[section]
\newtheorem{lem}[thm]{Lemma}
\newtheorem{prop}[thm]{Proposition}

\newtheorem{theorem}{Theorem}

\newtheorem{corollary}[theorem]{Corollary}
\makeatletter
\newtheorem*{rep@theorem}{\rep@title}
\newcommand{\newreptheorem}[2]{%
\newenvironment{rep#1}[1]{%
\def\rep@title{#2 \ref{##1}}%
\begin{rep@theorem}}%
{\end{rep@theorem}}}
\makeatother
\newreptheorem{theorem}{Theorem}

\newtheorem*{thm*}{Theorem}
\newtheorem*{lem*}{Lemma}
\newtheorem*{prop*}{Proposition}
\newtheorem*{cor*}{Corollary}
\newtheorem*{qu*}{Question}
\newtheorem*{dt*}{Definition and Theorem}
\newtheorem*{exmp*}{Example}
\newtheorem*{exmps*}{Examples}
\newtheorem*{dprop*}{Definition and Proposition}
\newtheorem*{conj*}{Conjecture}

\theoremstyle{definition}

\newtheorem*{defn*}{Definition}
\newtheorem*{not*}{Notation}

\theoremstyle{plain}
\newtheorem{rem}[thm]{Remark}
\newtheorem*{rem*}{Remark}

\DeclareMathOperator\FSym{FSym}
\DeclareMathOperator\Alt{Alt}
\DeclareMathOperator\FAlt{FAlt}
\DeclareMathOperator\Sym{Sym}

\DeclareMathOperator\Aut{Aut}

\DeclareMathOperator\supp{supp}
\newcommand{\sgn}{\mathrm{sgn}}

\begin{document}
\title{On the finite index subgroups of Houghton's groups}

\author{Charles Garnet Cox}
\address{School of Mathematics, University of Bristol, Bristol BS8 1UG, UK}
\email{charles.cox@bristol.ac.uk}

\thanks{}

\subjclass[2010]{20F05}

\keywords{generation, generation of finite index subgroups, structure of finite index subgroups, infinite groups, Houghton groups, permutation groups, highly transitive groups}
\date{\today}
\begin{abstract}
Houghton's groups $H_2, H_3, \ldots$ are certain infinite permutation groups acting on a countably infinite set; they have been studied, among other things, for their finiteness properties. In this note we describe all of the finite index subgroups of each Houghton group, and their isomorphism types. Using the standard notation that $d(G)$ denotes the minimal size of a generating set for $G$ we then show, for each $n\in \{2, 3,\ldots\}$ and $U$ of finite index in $H_n$, that $d(U)\in\{d(H_n), d(H_n)+1\}$ and characterise when each of these cases occurs.
\end{abstract}
\maketitle
\section{Introduction}
Introduced in \cite{Houghton} by Houghton, the Houghton groups have since attracted attention for their finiteness properties \cite{brown, finwreath}, their growth \cite{Hou2, conjgrowth}, their many interesting combinatorial features \cite{deadend, ConjHou, rinfinityHou2, Wiegold2} as well as other properties.
\begin{defn*} For $X\ne\emptyset$, let $\Sym(X)$ denote the group of all bijections on $X$. For $g\in \Sym(X)$, let $\supp(g):=\{x\in X\;:\;(x)g\ne x\}$, called the \emph{support} of $g$. Then $\FSym(X)=\{g\in\Sym(X)\;:\;|\supp(g)|<\infty\}$ and $\Alt(X)\le\FSym(X)$ consists of only the even permutations, meaning $[\FSym(X)\;:\;\Alt(X)]=2$.
\end{defn*}
We give a brief overview of these groups for our purposes; more detailed introductions can be found, for example, in \cite{ConjHou, Cox1}. We will use right actions throughout. We define $\N:=\{1, 2,\ldots\}$, let $G\le_f H$ denote that $G$ is a finite index subgroup of $H$, and for a group $G$ and $g, h\in G$ let $[g, h]:=g^{-1}h^{-1}gh$ and $g^h:=h^{-1}gh$.
\begin{defn*}
Let $n\in\{3, 4, \ldots\}$. Then the \emph{nth Houghton group}, denoted $H_n$, is generated by $g_2, \ldots, g_n\in \Sym(X_n)$ where $X_n=\{1, \ldots, n\}\times \N$ and for $k\in\{2, \ldots, n\}$,
\begin{equation}\label{thegi}
(i,m)g_k=\left\{\begin{array}{ll}(1, m+1) & \mathrm{if}\ i=1\ \mathrm{and}\ m\in\N\\ (1,1) & \mathrm{if}\ i=k\ \mathrm{and}\ m=1\\ (k, m-1) & \mathrm{if}\ i=k\ \mathrm{and}\ m\in\{2, 3, \ldots\}\\(i, m) & \mathrm{otherwise}.
\end{array}\right.
\end{equation}
\end{defn*}
For each $n\in\{3, 4, \ldots\}$ we have that $\FSym(X_n)\le H_n$. One way to see this is to first compute that $[g_2, g_3]=((1, 1), (1, 2))$, and then observe that any $2$-cycle with support in $X_n$ can be conjugated, using the elements $g_2, \ldots, g_n$, to $((1, 1), (1, 2))$. Furthermore, as observed in \cite{Wiegold2}, we have a short exact sequence of groups
 $$
1\longrightarrow \FSym(X_n)\stackrel{}{\longrightarrow} H_n\stackrel{\pi}{\longrightarrow} \Z^{n-1}\longrightarrow 1.
 $$
Here $\pi$ is induced by defining $\pi(g_i):=e_{i-1}$ for $i=2, \ldots, n$, where $e_i$ denotes the vector in $\Z^{n-1}$ with $i$th entry 1 and other entries 0.
\begin{defn*} The second Houghton group is generated by the two cycle $((1, 1), (1, 2))$ together with the element $g_2$, defined analogously to (\ref{thegi}) above. This is isomorphic to $\FSym(\Z)\rtimes\langle t\rangle$ where $t\in \Sym(\Z)$ sends each $z\in \Z$ to $z+1$.
\end{defn*}

There have been many papers with questions and results relating to the finite index subgroups of this family of groups, e.g.\ the questions on invariable generation in \cite{Wiegold2} and subsequent answers in \cite{invgenhou}, showing they all have solvable conjugacy problem in \cite{Cox1}, and also all have the $R_\infty$ property \cite{rinfinity, Cox2}. Some of these use the partial description of the finite index subgroups from \cite{Hou2}. We start by giving the first full description of them.

\begin{theorem} \label{mainthmA} Let $n\in \{3, 4, \ldots\}$ and $U\le_f H_n$. Then there exist $c_2, \ldots, c_n\in\N$ such that $\pi(U)=\langle c_2e_1, \ldots, c_ne_{n-1}\rangle$ and either $U$ is:
\begin{itemize}
\item[i)] equal to $\langle \FSym(X_n), g_i^{c_i}\;:\;i=2, \ldots, n\rangle$; or
\item[ii)] isomorphic to $\langle \Alt(X_n), g_i^{c_i}\;:\;i=2, \ldots, n\rangle$.
\end{itemize}
If $U\le_fH_2$, then there exists $c_2\in \N$ such that $\pi(U)=\langle c_2\rangle\le\Z$ and either (i) or (ii) above occurs or $U$ is equal to $\langle \Alt(X_2), ((1, 1), (1, 2)g_2)\rangle$.
\end{theorem}
Given $U, U'\le_fH_n$ such that $\pi(U)=\langle c_2e_1, \ldots, c_ne_{n-1}\rangle$ for $c_2, \ldots, c_n\in \N$ and $\pi(U')=\langle d_2e_1, \ldots, d_ne_{n-1}\rangle$ for $d_2, \ldots, d_n\in \N$, one might wonder when $U\cong U'$. Clearly any permutation of the constants $c_2, \ldots, c_n$ produces an isomorphism. By considering $\Aut(U)$ and $\Aut(U')$, it seems that this is the only way for the groups to be isomorphic. Our methods do allow us to obtain the following.
\begin{corollary}\label{corollaryA} Let $n \in \{2, 3, \ldots\}$ and $c_2, \ldots, c_n\in\N$. If $U\le_fH_n$ and $\pi(U)=\langle c_2e_1, \ldots, c_ne_{n-1}\rangle$, then either
\begin{itemize}
\item at least two of $c_2, \ldots, c_n$ are odd and $U=\langle g_2^{c_2}, \ldots, g_n^{c_n}, \FSym(X_n)\rangle$; or
\item $U$ is one of exactly $2^{n-1}+1$ specific subgroups of $H_n$.
\end{itemize} 
\end{corollary}
We then extend the work in \cite{infspread}, where the groups $\langle \Alt(X_2), g_2^c\rangle$ were shown to be $2$-generated for each $c\in \N$, by investigating the generation properties of each of these groups.

\begin{not*} For a finitely generated group $G$, let $d(G):=\min\{|S|\;:\;\langle S\rangle=G\}$.
\end{not*}

Every finite index subgroup of a finitely generated group is itself finitely generated. In the case of $\Z^n$, where $n\in \N$, we have that each $A\le_f \Z^n$ satisfies $d(A)=d(\Z^n)$. There has been much work related to this notion, e.g.\ \cite{othergenspaper} and papers verifying the Nielsen–Schreier formula for families of groups other than free groups. Our second theorem states that there is a close connection between the minimal number of generators of a Houghton group and its finite index subgroups. Note that $d(H_2)=2$ and $d(H_n)=n-1$ for $n\in\{3, 4, \ldots\}$. 

\begin{theorem} \label{mainthmB} If $U\le_f H_2$, then $d(U)=d(H_2)$. For $n\in \{3, 4 \ldots\}$ and $U\le_fH_n$, we have that $d(U)\in\{d(H_n), d(H_n)+1\}$. Furthermore, let $\pi(U)=\langle c_2e_1, \ldots, c_ne_{n-1}\rangle$. Then $d(U)=d(H_n)+1$ occurs exactly when both of the following conditions are met:
\begin{enumerate}[i)]
\item that $\FSym(X_n)\le U$; and
\item either one or zero elements in $\{c_2, \ldots, c_n\}$ are odd.
\end{enumerate}
\end{theorem}
Our proof involves providing a generating set. Theorem \ref{mainthmA} allows us to replace $U$ with either $\langle \FSym(X_n), g_2^{c_2}, \ldots, g_n^{c_n}\rangle$ or $\langle \Alt(X_n), g_2^{c_2}, \ldots, g_n^{c_n}\rangle$. Proposition \ref{genofGc} states that, in the second case, $d(U)=n-1$. Lemma \ref{setIn} and Lemma \ref{setIn2} combine to tell us that $\langle \FSym(X_n), g_2^{c_2}, \ldots, g_n^{c_n}\rangle=\langle \Alt(X_n), g_2^{c_2}, \ldots, g_n^{c_n}\rangle$ when there are distinct $i, j \in \{2, \ldots, n\}$ that are both odd. Thus the only remaining possibility is that $U=\langle \FSym(X_n), g_2^{c_2}, \ldots, g_n^{c_n}\rangle\ne \langle \Alt(X_n), g_2^{c_2}, \ldots, g_n^{c_n}\rangle$. Lemma \ref{needmoregens} shows that in this case $d(U)=d(H_n)+1$, by determining that the abelianization of $U$, with these conditions, is $C_2\times \Z^{n-1}$. This complete categorisation provides us with subgroups of the Houghton groups with constant minimal number of generators on finite index subgroups.
\begin{corollary} \label{corollaryB} Let $n\in \{3, 4, \ldots\}$ and define $G_{\bf 2}:=\langle \Alt(X_n), g_2^2, \ldots, g_n^2\rangle\le_f H_n$. If $U\le_f G_{\bf 2}$, then $d(U)=d(G_{\bf 2})=d(H_n)$.
\end{corollary}

\vspace{0.1cm}
\noindent\textbf{Acknowledgements.} I thank the anonymous referee and the editor for their excellent comments.

\section{The structure of finite index subgroups of $H_n$}
In this section we prove Theorem \ref{mainthmA}. The following is well known.
\begin{lem}\label{wellknown} Given $U\le_fG$, there exists $N\le_fU$ which is normal in $G$.
\end{lem}
Some structure regarding finite index subgroups of $H_n$ is known; the following is particularly useful to us.
\begin{lem} \cite[Prop. 2.5]{Cox2}\label{monolithic} Let $X$ be a non-empty set and $\Alt(X)\le G\le \Sym(X)$. Then $G$ has $\Alt(X)$ as a unique minimal normal subgroup.
\end{lem}
Let $n\in \{2, 3, \ldots\}$ and $U\le_f H_n$. By Lemma \ref{wellknown}, $U$ contains a normal subgroup of $H_n$ and so, by Lemma \ref{monolithic}, $\Alt(X_n)\le U$. Furthermore, $U\le_f H_n$ and so $\pi(U)\le_f\pi(H_n)$ where $\pi: H_n\rightarrow \Z^{n-1}$ sends $g_i$ to $e_{i-1}$ for $i=2, \ldots, n$. Hence there exist minimal $k_2, \ldots, k_n\in \N$ such that $\pi(g_i^{k_i})\in\pi(U)$ for each $i\in\{2, \ldots, n\}$. But $\pi(g_i^{k_i})=k_ie_{i-1}$, and so the preimage in $H_n$ of $\pi(g_i^{k_i})$ is $\{\sigma g_i^{k_i}\;:\;\sigma\in\FSym(X_n)\}$. If $\FSym(X_n)\le U$, then $g_2^{k_2}, \ldots, g_n^{k_n} \in U$. Otherwise $\FSym(X_n)\cap U=\Alt(X_n)$ and, for each $i\in\{2, \ldots, n\}$, either $g_i^{k_i}\in U$ or there exists $\omega_i\in \FSym(X_n)\setminus\Alt(X_n)$ such that $\omega_i g_i^{k_i}\in U$; since $\Alt(X_n)\le U$, we can specify that $\omega_i=((1, 1)\;(1, 2))$. Then, using the minimality of $k_2, \ldots, k_n$,
\begin{equation}\label{allfinind}
U=\langle g_2^{k_2}, \ldots, g_n^{k_n}, \FSym(X_n)\rangle\text{ or }U=\langle \epsilon_2g_2^{k_2}, \ldots, \epsilon_ng_n^{k_n}, \Alt(X_n)\rangle
\end{equation}
where each $\epsilon_i$ is either trivial or $((1, 1)\;(1, 2))$. We now describe the isomorphism type for these subgroups. To do this, we introduce two families of finite index subgroups in $H_n$ where $n\in \{2, 3, \ldots\}$.
\begin{not*} For any given $n\in\{2, 3, \ldots\}$ and any ${\bf c}=(c_2, \ldots, c_n)\in\N^{n-1}$, let $F_{\bf c}:=\langle \FSym(X_n), g_2^{c_2}, \ldots, g_n^{c_n}\rangle$ and $G_{\bf c}:=\langle \Alt(X_n), g_2^{c_2}, \ldots, g_n^{c_n}\rangle$.
\end{not*}
There are some ${\bf c}\in\N^n$ such that $G_{\bf c}=F_{\bf c}$, e.g.\ if $n\ne2$ and $c_2=\cdots=c_n=1$.
\begin{not*} Let $I_n:=\{{\bf c}=(c_2, \ldots, c_n)\in \N^n\;:\;G_{\bf c}\ne F_{\bf c}\}$.
\end{not*}
Lemma \ref{setIn} and Lemma \ref{setIn2} together describe $I_n$. This description, together with (\ref{allfinind}), yields Corollary \ref{corollaryA}.
\begin{not*} For $n\in \{2, 3, \ldots\}$ and $i\in \{1, \ldots, n\}$, let $R_i:=\{(i,m)\;:\;m\in\N\}\subset X_n$.
\end{not*}
\begin{prop} Let $n\in \{3, 4, \ldots\}$, ${\bf c}=(c_2, \ldots, c_n)\in I_n$, each $\epsilon_2, \ldots, \epsilon_n$ be either trivial or $((i, 1)\;(i, 2))$, and $U=\langle \Alt(X_n), \epsilon_ig_i^{c_i}\;:\; i=2, \ldots, n, \rangle\le_f H_n$. Then $U$ is isomorphic to $G_{\bf c}$.
\end{prop}
\begin{proof} If there are $i\ne j$ such that $c_i=c_j=1$, then $\FSym(X_n)\le  G_{\bf c}$ and so $G_{\bf c}= F_{\bf c}$, i.e.\ ${\bf c}\not\in I_n$.

Now consider if $c_i=1$ for some $i$ and that $\epsilon_i =((i, 1)\;(i, 2))$. Then $\epsilon_ig_i$ fixes $(i,1)$. Moreover, we have $(i, m)\epsilon_ig_i=(i,m-1)$ for all $m\ge3$, $(i,2)\epsilon_ig_i=(1, 1)$, $(1, m)\epsilon_i=(1,m+1)$ for all $m\ge1$, and that $\epsilon_ig_i$ fixes all other points in $X_n$. This allows us to relabel $X_n$ so that the point $(i,1)$ becomes part of another ray $j\in \{2, \ldots, n\}$.

If $c_i\ne 1$ and $\epsilon_i=((i,1)\;(i,2))$, then we can relabel $R_i$ as $R_i'$ by swapping the labels on the points $(i, 1)$ and $(i, 2)$ so that $\epsilon_ig_i^{c_i}$ acts on $R_i'\cup R_1$ in the same way as $g_i^{c_i}$ acts on $R_i\cup R_1$.
\end{proof}

\begin{lem} Let $k\in \{2, 3, \ldots\}$, $\epsilon=((1, 1)\;(1, 2))$, and $U=\langle \Alt(X_n), \epsilon g_2^k \rangle\le_f H_2$. Then $U$ is isomorphic to $\langle \Alt(X_2), g_2^k\rangle$.
\end{lem}
\begin{proof} In a similar way to the preceding proof, we can relabel $R_2$ as $R_2'$ (by swapping the labels on the points $(2, 1)$ and $(2, 2)$) so that $\epsilon g_2^k$ acts on $R_1\cup R_2'$ in the same way that $g_2^k$ acts on $R_1\cup R_2$.
\end{proof}
\begin{rem}
The preceding lemma may hold for $k=1$. Such an isomorphism would not be induced by a permutation of $X_2$, however: any element $g$ in \linebreak$\langle \Alt(X_2), ((1, 1)\;(1, 2)) t\rangle$ with $\pi(g)=1$ cannot have an infinite orbit equal to $X_2$.
\end{rem}

\section{Generation of the groups $F_{\bf c}$ and $G_{\bf c}$}
We start with the case of $H_2$, first dealing with the `exceptional' case.
\begin{lem} The group $\langle \Alt(X_2), ((1, 1)\;(1, 2))g_2\rangle$ is $2$-generated.
\end{lem}
\begin{proof} Note that $\langle \Alt(X_2), ((1, 1)\;(1, 2))g_2\rangle\cong \langle \Alt(\Z), (1\;2)t\rangle\le \FSym(\Z)\rtimes\langle t\rangle$ where $t: z\rightarrow z+1$ for all $z\in\Z$. Our aim will be to show that $S=\{(0\;1\;2), (0\;1)t\}$ generates $\langle \Alt(\Z), (1\;2)t\rangle$. To do this, we will use $\langle S\rangle$ to construct all 3-cycles $(0\;a\;b)$ with $0<a<b$. With these elements we can then also produce every 3-cycle of the form $(a\;b\;c)$ where $0<a<b<c$, by conjugating $(0\;a\;b)$ by $(0\;c\;2c)$. Then every 3-cycle in $\Alt(\N\cup\{0\})$, or its inverse, is accounted for by such elements. Conjugation by some power of $(0\;1)t$ yields every 3-cycle in $\Alt(\Z)$, and we recall that $\Alt(\Z)$ is generated by the set of all 3-cycles with support in $\Z$.

In addition to our first simplification, note that it is sufficient to show that $(0, 1, k+1)\in \langle(0\;1\;2), (0\;1)t\rangle$ for every $k\in \N$. This is because any 3-cycle $(0\;a\;b)$ with $0<a<b$ will be conjugate, by $(0\;1)t$, to an element of the form $(0, 1, k+1)$. We start with $\sigma_1:=(0, k, k+1)$. If $k=1$, then we are done. Otherwise, conjugate $\sigma_1^{-1}$ by $(0, k-1, k)$ to obtain $\sigma_2=(0, k-1, k+1)$. If $k>2$, conjugate $\sigma_2^{-1}$ by $(0, k-2, k-1)$ to obtain $\sigma_3$. Continuing in this way yields the result.
\end{proof}
The following results conclude the $H_2$ case. We include the proof here as we will adapt it for $F_{\bf c}$ and $G_{\bf c}$. The notation $\Omega_{k}:=\{1, \ldots, k\}$ is helpful for these results.
\begin{lem} \cite[Lem. 3.6]{infspread} \label{lem1} Let $k\in \{3, 4, \ldots\}$ and $t: z\rightarrow z+1$ for all $z\in \Z$. Then $\langle \Alt(\Z), t^k\rangle$ is generated by $\Alt(\Omega_{2k}) \cup\{t^k\}$.
\end{lem}
\begin{lem} \cite[Lem. 3.7]{infspread} \label{lem2} Let $k\in \N$ and $t: z\rightarrow z+1$ for all $z\in \Z$. Then $G_k:=\langle \Alt(\Z), t^k\rangle$ and $F_k:=\langle \FSym(\Z), t^k\rangle$ are 2-generated.
\end{lem}
\begin{proof} We start with $G_k$. We will show that we can find, for each $k\in \{3, 4, \ldots\}$, an $\alpha_k\in \Alt(\Z)$ such that $\langle t^k, \alpha_k\rangle$ contains $\Alt(\Z)$. Then, for $d\in \{1, 2, 3\}$, $\langle t^d, \alpha_6\rangle$ contains $\langle t^6, \alpha_6\rangle$ and so contains $\Alt(\Z)$. We can therefore fix some $k\in\{3, 4, \ldots\}$, meaning all $3$-cycles in $\Alt(\Omega_{2k})$ are conjugate.

Let $r=\binom{2k}{3}$ and let $\omega_1, \ldots, \omega_r$ be a choice of distinct $3$-cycles in $\Alt(\Omega_{2k})$ with $\omega_i\ne\omega_j^{-1}$ for every $i,j \in \{1, \ldots, r\}$. Thus $\langle \omega_1, \ldots, \omega_r\rangle=\Alt(\Omega_{2k})$. Set $\sigma_0=(1\;3)$ and $\sigma_{r+1}=(2\;3)$ and note that $(\sigma_{r+1}^{-1}\sigma_0^{-1}\sigma_{r+1})\sigma_0=(1\;2\;3)$. Now choose $\sigma_1, \ldots, \sigma_{r}\in \Alt(\Omega_{2k})$ so that for each $m\in\{1,\ldots, r\}$  we have $\sigma_m^{-1}(1\;2\;3)\sigma_m=\omega_m$.

Let $\alpha_k:=\prod_{i=0}^{r+1}t^{-2ik}\sigma_it^{2ik}$ and, for $m\in\{1,\ldots, r+1\}$, let $\beta_m:=t^{2mk}\alpha_k t^{-2mk}$. Then $\beta_{r+1}^{-1}\alpha_k^{-1}\beta_{r+1}\alpha_k=\sigma_{r+1}^{-1}\sigma_0^{-1}\sigma_{r+1}\sigma_0=(1\;2\;3)$ and, for $m\in\{1, \ldots, r\}$, we have that $\beta_m^{-1}(1\;2\;3)\beta_m=\omega_m$. Hence $\langle \alpha_k, t^k\rangle=G_k$.

We can now adapt the element $\alpha_k$ to an element $\gamma_k$ so, for each $k\in \{3, 4, \ldots\}$, we have that $\langle \gamma_k, t^k\rangle=F_k$. One way to do so is to set $\gamma_k:=(\sigma_0')(\prod_{i=1}^{r+1}t^{-2ik}\sigma_it^{2ik})$ where $\sigma_0'=(1\;3)(4\;5)$. This means that $[\sigma_{r+1}, \sigma_0']$ still equals $(1\;2\;3)$, and so $\Alt(\Z)\le \langle \gamma_k, t^k\rangle$, but that $\Alt(\Z)\cup\gamma_k(\Alt(\Z))=\FSym(\Z)\le \langle \gamma_k, t^k\rangle$ as well. 
\end{proof}
We now work with a fixed $n\in \{3, 4, \ldots\}$. The following well-known commutator identities will be helpful for a few of our remaining proofs.
\begin{equation}\label{comm1}
[a, bc]=a^{-1}(bc)^{-1}a(bc)=a^{-1}c^{-1}acc^{-1}a^{-1}b^{-1}abc=[a, c][a, b]^c
\end{equation}
\begin{equation}\label{comm2}
[ab, c]=(ab)^{-1}c^{-1}(ab)c=b^{-1}(a^{-1}c^{-1}ac)bb^{-1}c^{-1}bc=[a, c]^b[b, c]
\end{equation}
\begin{rem} \label{sgn} Let $a, b, c \in H_n$ and $\sgn$ denote the sign function on $\FSym(X_n)$. From the identities (\ref{comm1}) and (\ref{comm2}) above, we have that $\sgn([a, bc])=\sgn([a, b])\sgn([a, c])$ and $\sgn([ab, c])=\sgn([a, c])\sgn([b, c])$.
\end{rem}
Recall that ${\bf c}\in\N^{n-1}$ is in $I_n$ if and only if $G_{\bf c}\ne F_{\bf c}$.
\begin{lem} \label{setIn} Let ${\bf c}=(c_2, \ldots, c_n)\in\N^n$. We have that ${\bf c}\in I_n$ if and only if $[g_i^{c_i}, g_j^{c_j}]\in\Alt(X_n)$ for all $i, j \in \{2, \ldots, n\}$.
\end{lem}
\begin{proof} If there exist $i, j\in\{2, \ldots, n\}$ such that $[g_i^{c_i}, g_j^{c_j}]\not\in\Alt(X_n)$, then $G_{\bf c}$ contains an odd permutation and so $G_{\bf c}= F_{\bf c}$. Recall that $G_{\bf c}=\langle \Alt(X_n), g_2^{c_2}, \ldots, g_n^{c_n}\rangle$ and that we have a homomorphism $\pi: H_n\rightarrow \Z^{n-1}$ induced by sending $g_i$ to $e_{i-1}$ for $i=2, \ldots, n$. Then $\pi(G_{\bf c})\le_f\pi(H_n)=\Z^{n-1}$, and so $\pi(G_{\bf c})$ is free abelian of rank $n-1$. Define $S:=\{g_2^{c_2}, \ldots, g_n^{c_n}\}$ so that $\langle S\cup \Alt(X_n)\rangle = G_{\bf c}$ and $\pi(S)$ is a linearly independent set. Let $t_i:=\pi(g_i^{c_i})$ for $i=2, \ldots, n$. Thus
\begin{equation}\label{presentation}
\pi(\langle S\rangle)=\langle t_2, \ldots, t_n\mid R\rangle\text{ where }R=\{[t_i, t_j]\;:\;i\ne j\}.
\end{equation}

Now let $\alpha \in \langle S\rangle\cap \FSym(X_n)$. We will show that $\alpha \in \Alt(X_n)$, meaning that ${\bf c}\in I_n$. First, express $\alpha$ as a word $w$ in $S^{\pm1}$. Thus $\alpha=w(g_2^{c_2}, \ldots, g_n^{c_n})$ and $\pi(\alpha)=w(t_2, \ldots, t_n)$ for the same word $w$. But also $\pi(\alpha)=\underline{0}$ which, from (\ref{presentation}), means that $w(t_2, \ldots, t_n)$ is in the normal closure of $R$. Thus $w(t_2, \ldots, t_n)$ is a product of conjugates of powers of commutators in terms of $t_2, \ldots, t_n$. Then $\alpha=w(g_2^{c_2}, \ldots, g_n^{c_n})$ must also be a product of conjugates of powers of commutators in terms of $g_2^{c_2}, \ldots, g_n^{c_n}$. Our assumption that $[g_i^{c_i}, g_j^{c_j}]\in\Alt(X_n)$ for all $i, j \in \{2, \ldots, n\}$ together with Remark \ref{sgn} means that $\alpha \in \Alt(X_n)$.
\end{proof}
The following gives a clearer description of the set $I_n$.
\begin{lem} \label{setIn2} Let $i, j\in \{2, \ldots, n\}$ and $c_i, c_j\in \N$. Then $[g_i^{c_i}, g_j^{c_j}]\not\in\Alt(X_n)$ if and only if $c_i$ and $c_j$ are both odd. 
\end{lem}
\begin{proof} Note that $[g_i, g_j]\in\FSym(X_n)\setminus\Alt(X_n)$ whereas $[g_i^2, g_j^2], [g_i, g_j^2], [g_i^2, g_j] \in \Alt(X_n)$. Now, by repeatedly applying Remark \ref{sgn}, we can reduce $c_i$ and $c_j$ to be either 1 or 2, depending on whether they are odd or even respectively.
\end{proof}
Combining Lemma \ref{setIn} and Lemma \ref{setIn2}, we see that ${\bf c}=(c_2, \ldots, c_n)\in \N^{n-1}$ is in $I_n$ if and only if at least two of $c_2, \ldots, c_n$ are odd.
\begin{prop}\label{genofGc} Given ${\bf c}=(c_2, \ldots, c_n)\in \N^{n-1}$, we have that $d(G_{\bf c})=d(H_n)$.
\end{prop}
\begin{proof} We start with the case that ${\bf c}\in I_n$. By possibly relabelling the branches of $X_n$, set $c_3:=\max\{c_2, \ldots, c_n\}$. Thus $c_3\ne1$, as otherwise $c_2=c_3=1$ and ${\bf c}\not\in I_n$. Let $k:=2c_3$ and $\Omega^*:=\{(1, 1), \ldots, (1, 2k)\}$. Thus $2k\ge 8$ and all 3-cycles in $\Alt(\Omega^*)$ are conjugate in $\Alt(\Omega^*)$. We claim that $\Alt(X_n)\le \langle \Alt(\Omega^*), g_2^{c_2}, \ldots, g_n^{c_n}\rangle$. By Lemma \ref{lem1}, $\langle \Alt(\Omega^*), g_3^{c_3}\rangle$ is a subgroup of $H_n$ which generates $\langle \Alt(R_1\cup R_3), g_3^{c_3}\rangle \cong\langle \Alt(\Z), t^{c_3}\rangle$. Most importantly, $\langle \Alt(\Omega^*), g_3^{c_3}\rangle$ contains $\Alt(R_1)$. Now, given any $\sigma\in \FSym(X_n)$, there exists a word of the form $g_2^{d_2}\ldots g_n^{d_n}$ for some $d_i\in c_i\N$ which conjugates $\sigma$ to $\sigma'$, where $\supp(\sigma')\subset R_1$. Hence $\langle \Alt(\Omega^*), g_2^{c_2}, \ldots, g_n^{c_n}\rangle=G_{\bf c}$. Our next aim is to adapt the proof of Lemma \ref{lem2}. In particular, we will show that there exists a $\sigma\in \Alt(X_n)$ such that $\langle \sigma g_2^{c_2}, g_3^{c_3}, \ldots, g_n^{c_n}\rangle =G_{\bf c}$. For any $\omega\in \FSym(X_n)$, let $\sigma_\omega:=[\omega g_2^{c_2}, g_3^{2k}]$, which we can calculate using (\ref{comm2}) as
\begin{equation}\label{keyelement}
[\omega, g_3^{2k}]^{g_2^{c_2}}[g_2^{c_2}, g_3^{2k}]=g_2^{-c_2}(\omega^{-1}g_3^{-2k}\omega g_3^{2k})g_2^{c_2}[g_2^{c_2}, g_3^{2k}].
\end{equation}
Observe that $\supp([g_2^{c_2}, g_3^{2k}])\subseteq \{(1, 1), \ldots, (1, c_2+2k)\}$ and, since $k=2c_3>2c_2$, that $(1, 4k+1), (1, 4k+3)\not\in\supp([g_2^{c_2}, g_3^{k}])$. Let $\delta_0:=((1, 2k-c_2+1), (1, 2k-c_2+3))$ and note, using (\ref{keyelement}), that $\supp(\sigma_{\delta_0})\cap\{(1, m)\;:\;m=4k+2 \text{ or }m\ge4k+4\}=\emptyset$ and $\sigma_{\delta_0}$ swaps $(1, 4k+1)$ and $(1, 4k+3)$. We now restrict to those $\omega\in \Alt(X_n)$ with
$$\supp(\omega\delta_0)\subset \bigcup_{j\in\N}(\Omega^*)g_3^{-4jk}.$$
The following implications of this restriction are all important in relation to (\ref{keyelement}):
\begin{itemize}
\item $\supp(\sigma_\omega)\subset R_1\cup R_3$;
\item $\supp(\omega^{-1})\cap \supp(g_3^{-2k}\omega g_3^{2k})=\emptyset$;
\item $(x)g_2^{-c_2}\omega^{-1}g_2^{c_2}=(x)\omega^{-1}$ for all $x\in R_3$; and
\item $(x)g_2^{-c_2}(g_3^{-2k}\omega g_3^{2k})g_2^{c_2}=(x)g_3^{-2k}\omega g_3^{2k}$ for all $x\in R_3$.
\end{itemize}

Now, by mimicking the proof of Lemma \ref{lem2}, we can choose one such $\omega$, which we will denote by $\alpha$, so that $\Alt(R_1)\le\langle g_3^k, \sigma_{\alpha}\rangle$. Let $r=\binom{2k}{3}$ and $\omega_1, \ldots, \omega_r$ be a set of 3-cycles such that $\langle \omega_1, \ldots, \omega_r\rangle=\Alt(\Omega^*)$. Choose $\sigma_1, \ldots, \sigma_r\in\Alt(\Omega^*)$ such that $\sigma_i^{-1}((1, 1)\,(1, 2)\,(1, 3))\sigma_i=\omega_i$ for $i=1, \ldots, r$. Also set $\sigma_{r+1}:=((1, 2)\,(1, 3))$ and $\sigma_0:=((1, 1)\,(1, 3))$. Define $\delta_i:=g_3^{4ki}\sigma_ig_3^{-4ki}$ for $i=1, \ldots, r+1$ and $\alpha:=\delta_0\ldots\delta_{r+1}$. Define $\beta_m:=g_3^{-4mk}\sigma_\alpha g_3^{4mk}$ for $m=-1, \ldots, r+1$. Then, as in Lemma \ref{lem2}, we have that $[\beta_{r+1}, \beta_{-1}]=[\sigma_{d+1}, \sigma_0]$ and $\beta_i[\sigma_{d+1}, \sigma_0]\beta_i^{-1}=\omega_i$ for $i=1, \ldots, r$.

We end our proof with the case where ${\bf c}\not\in I_n$. From Lemma \ref{setIn} and Lemma \ref{setIn2}, there then exist $i, j \in \{2, \ldots, n\}$ such that $[g_i^{c_i}, g_j^{c_j}]\in\FSym(X_n)\setminus\Alt(X_n)$. Using the element $\sigma_\alpha\in \Alt(X_n)$ defined above, we claim that $\langle \sigma_\alpha g_2^{c_2}, g_3^{c_3}, \ldots, g_n^{c_n}\rangle = F_{\bf c}$. We note that if $[g_i^{c_i}, g_j^{c_j}]\in\FSym(X_n)\setminus\Alt(X_n)$, then, by Remark \ref{sgn}, so is $[\sigma g_i^{c_i}, g_j^{c_j}]$ for any $\sigma\in \FSym(X_n)$. Hence $\langle \sigma_\alpha g_2^{c_2}, g_3^{c_3}, \ldots, g_n^{c_n}\rangle$ contains both $\Alt(X_n)$ and an odd permutation, and so $d(G_{\bf c})=d(H_n)$.
\end{proof}
We now look at the final case of determining $d(F_{\bf c})$ where ${\bf c}\in I_n$.
\begin{lem}\label{commsubgroup} Let ${\bf c}\in I_n$. Then the commutator subgroup of $F_{\bf c}$ equals $\Alt(X_n)$.
\end{lem}
\begin{proof}
Note that $[\Alt(X_n), \Alt(X_n)]=\Alt(X_n)$. Let $\sigma\prod_{i\in I} g_i^{p_i}$ and $\omega \prod_{j\in J} g_j^{q_j}$ be in $F_{\bf c}$, where $I, J\subseteq\{2, \ldots, n\}$, $\{p_i\;:\;i\in I\}\subset \Z$, $\{q_j\;:\;j\in J\}\subset \Z$, and $\sigma, \omega\in\FSym(X_n)$. Then their commutator, using (\ref{comm1}), will be a product of conjugates of elements of the form $[\sigma\prod_{i\in I} g_i^{p_i}, \omega]$ or $[\sigma\prod_{i\in I} g_i^{p_i}, g_{k}^{q_{k}}]$ for some $k\in J$.  Applying (\ref{comm2}) to these elements will produce a product of conjugates of elements of the form
$$[\sigma, \omega], [g_l^{p_l}, \omega], [\sigma, g_k^{q_k}],\text{ or }[g_l^{p_l}, g_k^{q_k}]$$
for some $l\in I$ and $k\in J$. The first 3 of these are clearly in $\Alt(X_n)$. The 4th is also in $\Alt(X_n)$ from our assumption that ${\bf c}\in I_n$ together with Lemma \ref{setIn}. Hence, from Remark \ref{sgn}, a commutator of any two elements in $F_{\bf c}$ is in $\Alt(X_n)$.
\end{proof}
%Our aim is now to prove the converse of the preceding statement.
\begin{lem}\label{needmoregens} Let ${ \bf c}=(c_2, \ldots, c_n)\in I_n$. Then $d(F_{\bf c})=d(H_n)+1$.
\end{lem}
\begin{proof} Recall that $d(G_{\bf c})=n-1$. Including an element from $\FSym(X_n)\setminus\Alt(X_n)$ to a minimal generating set for $G_{\bf c}$ therefore yields a generating set of $F_{\bf c}$ of size $n$. So we need only show that, with the given hypothesis, no generating set of $F_{\bf c}$ of size $n-1$ exists. From Lemma \ref{commsubgroup}, the abelianization of $F_{\bf c}$ is $\Z^{n-1}\times C_2$. Thus $F_{\bf c}$ cannot be generated by $n-1$ elements.
\end{proof}

\bibliographystyle{amsalpha}
\def\cprime{$'$}
\providecommand{\bysame}{\leavevmode\hbox to3em{\hrulefill}\thinspace}
\providecommand{\MR}{\relax\ifhmode\unskip\space\fi MR }
\providecommand{\MRhref}[2]{%
 \href{http://www.ams.org/mathscinet-getitem?mr=#1}{#2}
}
\providecommand{\href}[2]{#2}

\end{document}

\title{Erratum and addendum}

\author{Charles Garnet Cox}
\address{School of Mathematics, University of Bristol, Bristol BS8 1UG, UK}
\email{charles.cox@bristol.ac.uk}

\thanks{}

\subjclass[2010]{20F05}

\keywords{generation, generation of finite index subgroups, structure of finite index subgroups, infinite groups, Houghton groups, permutation groups, highly transitive groups}
\date{\today}
%\begin{abstract}This note resolves an issue raised by Prof. Derek Holt for the paper \emph{On the finite index subgroups of Houghton's groups}.
%\end{abstract}
\maketitle

\setcounter{page}{8}
This note addresses an issue raised by Derek Holt that not all finite index subgroups of $\Z^n$ have the form $\langle c_1e_1, \ldots, c_ne_{n}\rangle$ for some constants $c_1, \ldots, c_n\in \N$. I am grateful to him for drawing my attention to these subgroups. In this note we show these additional subgroups can be dealt with using arguments of a similar nature to those in the original paper,  \cite{original}, and that a complete classification of when each $U\le_f H_n$ has $d(U)=n-1$ or $d(U)=n$ can still be obtained. We begin by giving reformulated statements of the original results with these subgroups in mind. The notation we use is introduced in \cite{original}. Recall that $\pi: H_n\rightarrow \Z^{n-1}$ is the homomorphism induced by  defining $\pi(g_i)=e_{i-1}$ for $i=2, \ldots, n$.

\begin{theorem} \label{mainthmA} Let $n\in \{3, 4, \ldots\}$ and $U\le_f H_n$. Then there exist $v_2, \ldots, v_{n}\in \Z^{n-1}$ such that $\langle v_2, \ldots, v_{n}\rangle \le_f \Z^{n-1}$ and $h_2, \ldots, h_n\in U$ where each $h_i\in \pi^{-1}(v_i)$ and either $U=\langle h_2, \ldots, h_n, \Alt(X_n)\rangle$ or $U=\langle h_2, \ldots, h_n, \FSym(X_n)\rangle$. Furthermore, if there exist $c_2, \ldots, c_n\in\N$ such that $\pi(U)=\langle c_2e_1, \ldots, c_ne_{n-1}\rangle$ then either $U$ is:
\begin{itemize}
\item[i)] equal to $\langle \FSym(X_n), g_i^{c_i}\;:\;i=2, \ldots, n\rangle$; or
\item[ii)] isomorphic to $\langle \Alt(X_n), g_i^{c_i}\;:\;i=2, \ldots, n\rangle$.
\end{itemize}
If $U\le_fH_2$, then there exists $c_2\in \N$ such that $\pi(U)=\langle c_2\rangle\le\Z$ and either (i) or (ii) above occurs or $U$ is equal to $\langle \Alt(X_2), ((1, 1), (1, 2))g_2\rangle$.
\end{theorem}
This theorem follows from the same principles in the original paper, but we give the details in Section \ref{1}. Corollary \ref{corollaryA} is not affected by the additional finite index subgroups of $\Z^n$, but is included for completeness.
\begin{corollary}\label{corollaryA} Let $n \in \{2, 3, \ldots\}$ and $c_2, \ldots, c_n\in\N$. If $U\le_fH_n$ and $\pi(U)=\langle c_2e_1, \ldots, c_ne_{n-1}\rangle$, then either
\begin{itemize}
\item at least two of $c_2, \ldots, c_n$ are odd and $U=\langle g_2^{c_2}, \ldots, g_n^{c_n}, \FSym(X_n)\rangle$; or
\item $U$ is one of exactly $2^{n-1}+1$ specific subgroups of $H_n$.
\end{itemize} 
\end{corollary}
Theorem \ref{mainthmB} can be rephrased so to only work with finite index subgroups of the specific form considered in the original paper, and so the original proofs apply to this rephrased result. 

\begin{theorem} \label{mainthmB} If $U\le_f H_2$, then $d(U)=d(H_2)$. Let $n\in \{3, 4 \ldots\}$ and $U\le_fH_n$ with $\pi(U)=\langle c_2e_1, \ldots, c_ne_{n-1}\rangle$. Then $d(U)\in\{d(H_n), d(H_n)+1\}$ and $d(U)=d(H_n)+1$ occurs exactly when both of the following conditions are met:
\begin{enumerate}[i)]
\item that $\FSym(X_n)\le U$; and
\item either one or zero elements in $\{c_2, \ldots, c_n\}$ are odd.
\end{enumerate}
\end{theorem}
Our fourth result no longer immediately follows from Theorem \ref{mainthmB}, since we need to consider the generation of the remaining finite index subgroups. We therefore name it Proposition \ref{corollaryB} rather than Corollary \ref{corollaryB}. The proof is in Section \ref{sec:corollaryB}.

\begin{proposition} \label{corollaryB} Let $n\in \{3, 4, \ldots\}$. If $U\le_f H_n$ and $\FSym(X_n)\cap U=\Alt(X_n)$, then $d(U)=n-1$. Thus for every $U\le_f G_{\bf 2}:=\langle \Alt(X_n), g_2^2, \ldots, g_n^2\rangle$ we have that $d(U)=d(G_{\bf 2})=d(H_n)$.
\end{proposition}
This result yields a complete characterisation of when $U\le_f H_n$ has $d(U)=n-1$ and when it has $d(U)=n$. For details, see Remark \ref{rem}. As a consequence of our adjusted approach, two new results are obtained. The proofs are contained in sections \ref{sec:new:propB1} and \ref{2}, respectively.
\begin{proposition}\label{new:propB1}
Let $U=\langle t^k, \Alt(\Z)\rangle \le_f H_2$ and $p\in\N\setminus\{1\}$. Then there exists $\omega\in \FSym(\Z)$ of order $p$ such that $\langle t^k, \omega\rangle=U$. If $U=\langle t^k, \FSym(\Z)\rangle \le_f H_2$, then such an $\omega$ exists if and only if $p$ is even.
\end{proposition}
\begin{proposition} \label{corollaryB2} Let $n\in \N$ and $U\le_f H_n$ with $\pi(U)=\langle v_2, \ldots, v_n\rangle$. Then for any choice of $h_2, \ldots, h_n\in U$ with $\pi(h_i)=v_i$ for $i=2, \ldots, n$, there exists an $\omega \in \FSym(X_n)$ such that $\langle \omega, h_2, \ldots, h_n\rangle=U$.
\end{proposition}

\vspace{0.1cm}
\noindent\textbf{Acknowledgements.} I thank Derek Holt for identifying and communicating the additional finite index subgroups, the editor for his help with submitting this note, and the anonymous referee for their helpful comments. 

\section{Additional details for Theorem \ref{mainthmA}}\label{1}
Two well known results (used in the original paper) give us Theorem \ref{mainthmA}.
\begin{lem}\label{wellknown} Given $U\le_fG$, there exists $N\le_fU$ which is normal in $G$.
\end{lem}
\begin{lem} \cite[Prop. 2.5]{Cox2}\label{monolithic} Let $X$ be a non-empty set and $\Alt(X)\le G\le \Sym(X)$. Then $G$ has $\Alt(X)$ as a unique minimal normal subgroup.
\end{lem}
\begin{proof}[Proof of Theorem \ref{mainthmA}.] Let $n\in \{2, 3, \ldots\}$ and $U\le_f H_n$. By Lemma \ref{wellknown}, $U$ contains a normal subgroup of $H_n$ and so, by Lemma \ref{monolithic}, $\Alt(X_n)\le U$. Furthermore, $U\le_f H_n$ and so $\pi(U)\le_f\pi(H_n)$. Thus $\pi(U)=\langle v_2, \ldots, v_n\rangle\le_f\Z^{n-1}$, meaning that $U$ contains some elements $h_2, \ldots, h_n$ where $h_i\in \pi^{-1}(v_i)$ for $i=2, \ldots, n$. Note that for any $v\in \Z^{n-1}$, we can describe $\pi^{-1}(v)$. Given $v=\sum_{j=2}^{n} d_je_{j-1}$ with $d_2, \ldots, d_n\in\Z$, define $\hat{v}:=\prod_{j=2}^ng_j^{d_j}$. With this notation, for $i=2, \ldots, n$, we have
\begin{equation*}\label{preimage}
\pi^{-1}(v_i)=\{\sigma \hat{v}_i\;:\;\sigma\in\FSym(X_n)\}.
\end{equation*}
If $\FSym(X_n)\le U$, then $U=\langle \hat{v}_2, \ldots, \hat{v}_n, \FSym(X_n)\rangle$. If $\FSym(X_n)\cap U=\Alt(X_n)$, then we can specify that $h_i$ is $\hat{v}_i$ or $\epsilon\hat{v}_i$ where $\epsilon=((1, 1)\;(1, 2))$, depending on whether $\hat{v}_ih_i^{-1}$ is an even or odd permutation respectively. Note that the possibility of $\epsilon\hat{v}_i, \hat{v}_i\in U$ is excluded since $\FSym(X_n)\not\le U$ means $\epsilon\not\in U$. Hence $U=\langle h_2, \ldots, h_n, \Alt(X_n)\rangle$ in this case.
\end{proof}
Note if we assume that $\pi(U)=\langle c_2e_1, \ldots, c_ne_{n-1}\rangle$ for some $c_2, \ldots, c_n\in \N$, then
\begin{equation*}\label{allfinind}
U=\langle g_2^{k_2}, \ldots, g_n^{k_n}, \FSym(X_n)\rangle\text{ or }U=\langle \epsilon_2g_2^{k_2}, \ldots, \epsilon_ng_n^{k_n}, \Alt(X_n)\rangle
\end{equation*}
where each $\epsilon_i$ is either trivial or equal to $((1, 1)\;(1, 2))$. This gives us Corollary \ref{corollaryA}.

\section{Proving Proposition \ref{corollaryB}}\label{sec:corollaryB}
We work with a fixed $n\in \{3, 4, \ldots\}$. To prove the proposition, it is sufficient to show that if $U\le_f H_n$ and $\FSym(X_n)\cap U=\Alt(X_n)$, then $U$ is $n-1$ generated.

Recall that $\pi(U)=\langle v_2, \ldots, v_n\rangle\le_f\Z^{n-1}$. Let $i\in \{2, \ldots, n\}$. The support of any preimage of $v_i$ has infinite intersection with at least 2 rays of $X_n$. Since $n\ge 3$, we therefore have $2(n-1)>n$ and so there exist distinct $j, j'$ such that the supports of the preimages of $v_j$ and $v_{j'}$ have infinite intersection with the same branch. By relabelling, we will assume this is $R_1:=\{(1, m) : m\in \N\}\subset X_n$ and that $j=2$ and $j'=3$. For $k\in \{2, \ldots, n\}$, let $h_k:=\hat{v}_k$ with the notation from Section \ref{1}. That is, if $v_k=\sum_{j=2}^{n} d_je_{j-1}$ with $d_2, \ldots, d_n\in\Z$, set $\hat{v}_k:=\prod_{j=2}^ng_j^{d_j}$. After possibly exchanging $v_2$ with $-v_2$ and $v_3$  with $-v_3$, we can assume that the first component of $\pi(h_2)$ and of $\pi(h_3)$ are positive.
 We claim that there exists $\omega\in\FSym(X_n)$ such that $\langle \omega h_2, h_3, \ldots, h_n\rangle=U$. Furthermore, we can specify that $\supp(\omega)\subset R_1$.

There are some key properties about the restrictions we have imposed thus far. Fix an $h\in \{h_2, \ldots, h_n\}$ so that $h=\prod_{j=2}^ng_j^{d_j}$ for some constants $d_2, \ldots, d_n$. We note that every orbit of $h$ is either infinite or has size 1. Let $i\in \{2, \ldots, n\}$ and $x\in R_i$. Then $(x)h=x$ if and only if $d_i=0$. If $d_i\ne 0$, then $\{(x)h^k : k\in \Z\}$ is infinite. Now, take any $y\in R_1$ and let $\mathcal{O}_y:=\{(y)h^k : k\in \Z\}$. If $\mathcal{O}_y\cap R_j\ne \emptyset$ for some $j\in \{2, \ldots, n\}$, then from the form of $h$ we must have $d_j\ne0$ and, from our first case, $\mathcal{O}_y$ is infinite. Finally, let $y=(1, m)$ and assume that $\mathcal{O}_y\subseteq R_1$. Then $(1, m)h=(1, m+\sum_{j=2}^nd_j)=(1, m)$ since $\mathcal{O}_y\subseteq R_1$ implies that $\sum_{j=2}^nd_j=0$.

Note that there is a $d\in \N$ such that $d\Z^{n-1}\le \pi(U)$. Hence for each $i=2, \ldots, n$ there is a $\sigma_i\in \FSym(X_n)$ such that $\sigma_i g_i^d \in\langle h_2, \ldots, h_n\rangle$. Elements of $H_n$ act `eventually as a translation' meaning that $\sigma_i g_i^d$ has $d$ infinite orbits, infinitely many fixed points, and then possibly some finite orbits (see \cite[Lem.\ 2.3]{infspread} for details and also note from this lemma that $R_1\cup R_i$ and $\supp(\sigma_i g_i^d)$ have finite symmetric difference). We can therefore choose some $m_2, \ldots, m_n\in \N$ so that $(\sigma_i g_i^d)^{m_i}$ has $m_id$ infinite orbits and no finite orbits other than fixed points. For $i=2, \ldots, n$ let $f_i:=(\sigma_i g_i^d)^{m_i}$ and $F_i:=(R_1\cup R_i)\setminus\supp(f_i)$, which is finite, and set $F:=\cup_{i=2}^nF_i$.
 
Our aim is therefore to move points in $F$ to $X_n\setminus F$. This can be achieved with a word of the form $(\omega h_2)^{k_2}h_3^{k_3}\cdots h_n^{k_n}$. Let $Y_2$ consist of points in $F\setminus R_1$ lying on an infinite orbit of $\omega h_2$. Then let $Y_i:=F\cap \supp(h_i)$ for $i=3, \ldots, n$, define $A_j:=\bigcup_{i=2}^{j-1} Y_i$ for $j=3, \ldots, n$, and let $Z_i:=Y_i\setminus A_{i}$ for $i=3, \ldots, n$. Choose $k_2', \ldots, k_n'$ such that $(Y_2)(\omega h_2)^{k_2'}\subset X_n\setminus F$ and $(Z_i)h_i^{k_i'}\subset X_n\setminus F$. Then for $j=2, \ldots, n$ pick $|k_j|\ge |k_j'|$ where $k_j$ and $k_j'$ have the same sign so that
$$(Y_i)(\omega h_2)^{k_2}\prod_{j=3}^nh_j^{k_j}\subset X_n\setminus F\text{ for }i=2, \ldots, n.$$
Finally use $f_2, \ldots, f_n$ to move the points in $X_n\setminus F$ to $R_1\setminus F$. For $j\in \N$ let $R_1^{(j)}:=\{(1, m) : m>j\}\subset X_n$. Then the above shows, for any $j\in \N$ and for any $\omega\in \FSym(X_n)$ with $\supp(\omega)\subset R_1$, that
\begin{equation}\label{eqn:reducing}
\langle \omega h_2, h_3, \ldots, h_n, \Alt(R_1^{(j)})\rangle=\langle \omega h_2, h_3, \ldots, h_n, \Alt(X_n)\rangle.
\end{equation}
Thus we need only generate $\Alt(R_1^{(j)})$ for some $j\in \N$ in order to generate $U$.

Note that there exist $d, d'\in \N$ such that the first component of $\pi(f_3^{2d'})$ and $\pi(h_3^d)$ are equal and that this lies in $8\N$. Denote this first component by $k$. We now compute that $[\omega h_2, h_3^d]=h_2^{-1}[\omega, h_3^d]h_2[h_2, h_3^d]$ where $[g, h]:=g^{-1}h^{-1}gh$ and so let $[h_2, h_3^d]=:\tau\in \FSym(X_n)$. Our aim is to choose $\omega$ so that $\langle h_2^{-1}[\omega, h_3^d]h_2\tau, h_3^d\rangle$ contains $\Alt(\supp(h_3^d))$. Since we can choose $\omega$ freely, it is sufficient to show that $\langle [\omega, h_3^d]\tau, h_3^d\rangle$ contains $\Alt(\supp(h_3^d))$. We can do so by mimicking the proof (from the original paper) of the following lemma.
\begin{lem} \cite[Lemma 3.3]{original} \label{lem2} Let $k\in \N$ and $t: z\rightarrow z+1$ for all $z\in \Z$. Then $G_k:=\langle \Alt(\Z), t^k\rangle$ and $F_k:=\langle \FSym(\Z), t^k\rangle$ are 2-generated.
\end{lem}
Note for $m\ge 8$ that $\langle (1\;2\;3), (2\;3\;4), \ldots, (m-2\;m-1\;m)\rangle=\Alt(\{1, \ldots, m\})$. Also, by \cite[Lemma 3.6]{infspread}, we have that
\begin{equation}\label{eqnaim}
\Alt(\{1, \ldots, m\})\leqslant \langle \alpha, t^{2m}\rangle\Rightarrow\Alt(\Z)\leqslant\langle \alpha, t^{m}\rangle.\end{equation}

Our aim is to find an $\omega\in \FSym(\Z)$ such that $\langle [\omega, h_3^d]\tau, h_3^d\rangle$ contains $\Alt(\supp(h_3^d))$. From the preceding paragraph, it is sufficient to show that $\langle [\omega, h_3^d]\tau, h_3^d\rangle$ contains certain $3$-cycles. Let $\tau$ have order $p\in \N$. Choose $p'\in \N$ with $p\vert p'$ and $p'-1\ge k+2$. We will restrict ourselves to those $\omega\in \FSym(X_n)$ such that:
\begin{itemize}
\item $\supp(\omega)\subset R_1$;
\item $\supp(h_3^{-d}\omega h_3^d)\cap \supp(\omega)=\emptyset$;
\item $\supp([\omega, h_3^d])\cap \supp(\tau)=\emptyset$; and
\item $\omega$ consists only of cycles of length $p'-1$.
\end{itemize}
Thus $([\omega, h_3^d]\tau)^{p'}=[\omega, h_3^d]$. Identify $R_1=\{(1, m) : m\in\N\}$ with $\N$ using the bijection $(1, m)\mapsto m$. Define $\sigma_1, \sigma_2\in\FSym(\N)<\FSym(\Z)$ so that for $i=1, 2$
$$\sigma_i: \max\{\supp(\sigma_i)\}\mapsto \min\{\supp(\sigma_i)\}\text{ and }\sigma_i(x)>x\text{ for all other }x\in \supp(\sigma_i).$$ We note that $p'-1=ak+b+2$ for some $a \in \N$ and $b\in \{0, \ldots, k-1\}$. Then let
$$\supp(\sigma_i):=\left(\bigcup_{i=0}^{a+1}\{2ki+1, \ldots, 2ki+k\}\right)\setminus A_i$$
where $$A_1:=\{3, \ldots, k\}\cup\{2k+1+b, \ldots, 3k\}$$ and $$A_2:=\{1+b, \ldots, 2k\}\cup\{2(a+1)k+1\}\cup\{2(a+1)k+4, \ldots, 2(a+1)k+k\}.$$
For $i=1, 2$, $\sigma_i$ can be seen as an element moving points within $a+2$ blocks, each of size $k$. Let $t: z\mapsto z+1\in \Sym(\Z)$. By construction $t^{-k}\sigma_1t^{k}$ has support disjoint from $\sigma_1$. Note that $\supp[\sigma_1, t^{k}]\subset \{1, 2, \ldots, 2(a+2)k\}\subset\Z$. In light of this, define $\sigma_2':=t^{-2(a+2)k}\sigma_2t^{2(a+2)k}$. Then $\supp([\sigma_2', t^{k}])\cap\supp([\sigma_1, t^{k}])=\emptyset$ and also $[\sigma_1\sigma_2', t^{k}]=[\sigma_1, t^{k}][\sigma_2', t^{k}]$. We first work with the element $\omega':=\sigma_1\sigma_2'$, and will show that there is a $c\in \N$ such that $\omega:=t^{-c}\omega't^c$ achieves our aim. Let $\alpha:=[\omega', t^k]$ and $q:=2(a+1)k+1$. Direct computation shows that $\max(\supp(\alpha))=q+2$ and $\alpha: q-k+1\mapsto q+1$ and $q+1\mapsto q+2$. Also $\min(\supp(\alpha))=1$ and so with $\beta:=t^{-(4(a+1)+2)k}\alpha t^{(4(a+1)+2)k}$ we see that $\supp(\alpha)\cap\supp(\beta)=\{q+1\}$. Similarly, $\min(\supp(\beta))=q$ and $\beta^{-1}: q \mapsto q+1$ and $q+1\mapsto q+2k$. Then $[\beta, \alpha]$ is the 3-cycle $(q\; q+1\; q+2)$ in $\FSym(\Z)$. Furthermore, for $i=1, \ldots, k-2$,
$$(t^{-2k}\alpha t^{2k})^{-i}(q, q+1, q+2)(t^{-2k}\alpha t^{2k})^{i}=(q+i, q+1+i, q+2+i).$$
Hence, from (\ref{eqnaim}), $\langle \alpha, t^{k/2}\rangle$ contains $\Alt(\Z)$.

We now work with $\omega'$ as an element in $\FSym(R_1)< \FSym(X_n)$ by using the bijection $m\mapsto (1, m)$ between $\N$ and $R_1$. Fix a $c\in k\N$ so that $\omega:=t^{-c}\sigma_1\sigma_2't^c$ has $\supp(\omega)\cap\supp(\tau)=\emptyset$ and
$$(1, x)f_3^{2d'}=(1, x)h_3^d=(1, x+k)\text{ for all }x\ge \min\{\supp(\omega)\}.$$ 
Then $f_3^{2d'}$ can take the role of $t^k$ in our above calculations, and the elements corresponding to $\alpha$ and $\beta$ above lie in $\langle [\omega, h_3^d], h_3^d\rangle\le \langle [\omega, h_3^d]\tau, h_3^d\rangle$. Thus $\langle f_3^{d'}, [\omega, h_3^d]\rangle$ contains $\Alt(\{(1, q), \ldots, (1, q+k-1)\}$ and also $\Alt(R_1^{(q)})$. Equation (\ref{eqn:reducing}) then implies the result.

\begin{rem}\label{rem} As in the original paper, it follows that $U$ requires $n$ generators exactly when $\FSym(X_n)\le U$ but $\FSym(X_n)\not\le [U, U]$, where $[U, U]$ denotes the commutator subgroup of $U$. We can see this as follows. Let $\pi(U)=\langle v_2, \ldots, v_n\rangle$ and $h_i\in \pi^{-1}(v_i)$ for $i=2, \ldots, n$. Given $g, h\in U$, we can find $\sigma_1$, $\sigma_2$ from $\FSym(X_n)$ such that $\sigma_1g=\prod_{i=2}^nh_i^{c_i}$ and $\sigma_2h=\prod_{i=2}^nh_i^{d_i}$ for some $c_2, d_2, \ldots, c_n, d_n\in \Z$. If $[h_i, h_j]\in \Alt(X_n)$ for all $i, j\in \{2, \ldots, n\}$, then applying the standard commutator identities (used in the original paper) means that $[g, h]$ is in $\Alt(X_n)$. This applies to any $g, h\in U$, and so $[U, U]\cap  \FSym(X_n)=\Alt(X_n)$ and $U$ has abelianization $\Z^{n-1}\times C_2$. Therefore $U$ cannot be $n-1$ generated if this is the case, but it will be $n$ generated. Otherwise, we have two cases. If $\FSym(X_n)\cap U=\Alt(X_n)$, then Proposition \ref{corollaryB} applies. If $[U, U]\cap  \FSym(X_n)=\FSym(X_n)$, then there exist $i, j\in \{2, \ldots, n\}$ such that $[h_i, h_j]\in \FSym(X_n)\setminus\Alt(X_n)$. We can then use the element $\omega$ from the proof of Proposition \ref{corollaryB} so that $\langle \omega h_2, h_3, \ldots, h_n\rangle$ contains $\Alt(X_n)$. But $\langle \omega h_2, h_3, \ldots, h_n\rangle$ must also contain $\FSym(X_n)$, since $$\{[h_i, h_j] : i=3, \ldots, n, j=2, \ldots, n\}\cup\{[\omega h_2, h_j] : j=2, \ldots, n\}$$ contains an element from $\FSym(X_n)\setminus\Alt(X_n)$.
\end{rem}

\section{Proving Proposition \ref{new:propB1}}\label{sec:new:propB1}

Choose $p\in\{2, 3, \ldots\}$ and $k\in \N$. Our aim will be to construct $\omega\in \Alt(\Z)$, consisting only of $p$-cycles, such that $\Alt(\Z)\le \langle \omega, t^{k}\rangle$. With this aim in mind, observe that we only need to show that $\Alt(\Z)\le \langle \omega, t^{ak}\rangle$ for some $a\in \N$, and so without loss of generality can assume that $k\ge 8$. We again use the idea of the proof of Lemma \ref{lem2}. Let $\Omega:=\{1, 2, \ldots, kp\}$. Recall that $\langle(1\;2\;3),\ldots, (1\;2\;kp)\rangle =\Alt(\Omega)$ which is a routine calculation within the proof of \cite[Lemma 3.6]{infspread}. Define $\sigma_0:=(1\;2\;\ldots\;p)$, $\sigma_i:=t^{-i}(1\;2\;\ldots\;p)t^{i}$ for $i=1, \ldots, kp-p$, and $\sigma_{kp-p+1}:=\sigma_{p-1}$. Note that these are all elements of $\FSym(\Omega)\le \FSym(\Z)$, and also for $i=1, \ldots, kp-p$ that $\sigma_i$ conjugates $(1\;2\;3)$ to $(1\;2\;3+i)$. Now let $n:=kp-p+1$ and define
$$\omega:=\prod_{i=0}^{n}t^{-kpi}\sigma_it^{kpi}$$
which is an element of $\Alt(\Z)$ by construction. Furthermore, with $\omega':=t^{kpn}\omega t^{-kpn}$ we see that $[\omega', \omega^{-1}]=(p+1\;p\;p-1)$. We can then conjugate the inverse of this element by $t^{-kpi}\omega t^{kpi}$ for various $i\in \N$ to obtain $\{(1\;2\;3),\ldots, (1\;2\;kp)\}$ which generate $\Alt(\Omega)$. Finally, 
$$\langle \Alt(\Omega), t^k\rangle\ge \langle \Alt(\{1, \ldots, 2k\}), t^k\rangle=\langle \Alt(\Z), t^k\rangle$$
by \cite[Lemma 3.6]{infspread}. If $p$ is even, then we can modify $\omega$ to include an extra orbit of length $p$ so that $\langle \omega, t^k\rangle =\langle \FSym(\Z), t^k\rangle$.

\section{Proving Proposition \ref{corollaryB2}}\label{2}
This proof again takes inspiration from Lemma \ref{lem2} above. Recall that $\pi(U)=\langle v_2, \ldots, v_n\rangle\le_f\Z^{n-1}$. We now let $h_2, \ldots, h_n$ denote arbitrary elements of $U$ satisfying $\pi(h_i)=v_i$ for $i=2, \ldots, n$. We will show that there exists $\omega\in \FSym(X_n)$ such that $U=\langle \omega, h_2, \ldots, h_n\rangle$. 

In the same way as the previous section, we can choose some $m_2, \ldots, m_n\in \N$ so that $(\sigma_i g_i^d)^{m_i}\in \langle h_2, \ldots, h_n\rangle$ has $m_id$ infinite orbits and no finite orbits. And again for $i=2, \ldots, n$ we let $f_i:=(\sigma_i g_i^d)^{m_i}$, $F_i:=(R_1\cup R_i)\setminus\supp(f_i)$, $F:=\cup_{i=2}^nF_i$, and $T:=\supp(f_2)$. Note that $(R_1\setminus F)\subset T$.

Produce a bijection $\phi: T\to \Z$ such that $\phi(xf_2)=\phi(x)t^{m_2d}$ for all $x\in T$. We can then apply Lemma \ref{lem2} to produce an element of $\FSym(X_n)$ which with $f_2$ generates a group containing either $\FSym(T)$ or $\Alt(T)$, depending on our preference. We can therefore choose $\tau\in \FSym(X_n)$ so that $\langle f_2, \tau\rangle\cap \FSym(X_n)=U\cap \FSym(T)$. Let $r$ denote the order of $\tau$. Choose a prime $q$ larger than $\max\{r, 2|F|\}$, and let $\omega_1\in \Alt(X_n)$ consist of a single orbit of length $q$ with $F\subset\supp(\omega_1)\subset R_1\cup F$ and $(F)\omega_1\subset R_1\setminus F$. Now conjugate $\tau$ by a suitable power of $f_2$ to obtain $\omega_2$, where $\supp(\omega_2)$ is contained in $R_1$ and has trivial intersection with $\supp(\omega_1)$. Clearly we have $\langle f_2, \omega_2\rangle=\langle f_2, \tau\rangle$. Set $\omega:=\omega_1\omega_2$. Using that $\omega_1$ and $\omega_2$ have disjoint supports, $\omega^r=\omega_1^r$. Since $\omega_1$ has prime order, $\langle \omega^r\rangle=\langle \omega_1^r\rangle=\langle \omega_1\rangle$. Thus $\langle \omega\rangle=\langle \omega_1, \omega_2\rangle$. Note, by construction, that any finite subset of $X_n$ can be moved into $R_1\setminus F$ by using an element of $\langle \omega_1, f_2, \ldots, f_n\rangle$. (First use powers of the elements $f_2, \ldots, f_n$ to send all points in $X_n\setminus F$ to $R_1\setminus \supp(\omega_1)$, and then use $\omega_1$ to send any points in $F$ to $R_1\setminus F$.) But then $\langle \omega_1, h_2, \ldots, h_n\rangle$ contains $\langle \omega_1, f_2, \ldots, f_n\rangle$, and so also has this property. Thus for any $\sigma\in\FSym(X_n)$, there exists $g\in\langle \omega_1, h_2, \ldots, h_n\rangle$ such that $\supp(g^{-1}\sigma g) \subset T\cap R_1$, i.e.,
$$\Alt(X_n)\le\langle \omega_1, h_2, \ldots, h_n, \Alt(T)\rangle\text{ and }\FSym(X_n)\le\langle \omega_1, h_2, \ldots, h_n, \FSym(T)\rangle.$$
Hence $\langle \omega, h_2, \ldots, h_n\rangle=\langle \omega_1, \omega_2, f_2, h_2, \ldots, h_n\rangle=U$, as required.

\bibliographystyle{amsalpha}
\def\cprime{$'$}
\providecommand{\bysame}{\leavevmode\hbox to3em{\hrulefill}\thinspace}
\providecommand{\MR}{\relax\ifhmode\unskip\space\fi MR }
\providecommand{\MRhref}[2]{%
 \href{http://www.ams.org/mathscinet-getitem?mr=#1}{#2}
}
\providecommand{\href}[2]{#2}